\documentclass{amsart}

\usepackage{amsthm,amsmath,amssymb,amscd,amsfonts,latexsym}

\theoremstyle{plain}
\newtheorem{theorem}{Theorem}[section]

\newtheorem{proposition}[theorem]{Proposition}

\newtheorem{corollary}[theorem]{Corollary}
\newtheorem{def-thm}[theorem]{Definition-Theorem}
\newtheorem{lemma}[theorem]{Lemma}

\theoremstyle{definition}

\newtheorem{remark}[theorem]{Remark}

\newtheorem{conjecture}[theorem]{Conjecture}

\newtheorem*{acknowledgement}{Acknowledgement}

\newcommand{\PP}{\mathbb{P}}

\newcommand{\RR}{\mathbb{R}}

\newcommand{\CC}{\mathbb{C}}
\newcommand{\QQ}{\mathbb{Q}}

\newcommand{\OO}{{\mathcal O}}

\newcommand{\dra}{\dashrightarrow}

\DeclareMathOperator{\Ric}{Ric}

\DeclareMathOperator{\kod}{kod}

\begin{document}

\title[Semi-negative holomorphic sectional curvature]{K\"ahler manifolds of semi-negative\\ holomorphic sectional curvature}

\begin{abstract} 
In an earlier work, we investigated some consequences of the existence of a K\"ahler metric of negative holomorphic sectional curvature on a projective manifold. In the present work, we extend our results to the case of semi-negative (i.e., non-positive) holomorphic sectional curvature. In doing so, we define a new invariant that records the largest codimension of maximal subspaces in the tangent spaces on which the holomorphic sectional curvature vanishes. Using this invariant, we establish lower bounds for the nef dimension and, under certain additional assumptions, for the Kodaira dimension of the manifold. In dimension two, a precise structure theorem is obtained.
\end{abstract}

\author{Gordon Heier, Steven S.~Y.~Lu, Bun Wong}

\address{Department of Mathematics\\University of Houston\\4800 Calhoun Road, Houston, TX 77204\\ USA\\ }
\email{heier@math.uh.edu}

\address{D\'epartment de Math\'ematiques\\Universit\'e du Qu\'ebec \`a Montr\'eal\\C.P. 8888\\Succursale Centre-Ville\\Montr\'eal, Qc H3C 3P8\\ Canada\\ }
\email{lu.steven@uqam.ca}

\address{Department of Mathematics\\UC Riverside\\900 University Avenue\\ Riverside, CA 92521\\ USA\\ }
\email{wong@math.ucr.edu}


\thanks{The first named author is partially supported by the National Security Agency under Grant Number H98230-12-1-0235. The United States Government is authorized to reproduce and distribute reprints notwithstanding any copyright notation herein.}

\maketitle

\section{Introduction} 

One of the most basic questions posed by S.-T.~Yau concerning the geometry of a projective (or compact) K\"ahler manifold is the relationship between its holomorphic sectional curvature and its Ricci curvature. The former plays a key role in classical complex geometry (recall for example the constant holomorphic sectional curvature characterization of quotients of ${\mathbb B}^n$,
${\mathbb C}^n$ and ${\mathbb P}^n)$, while the latter
is the keystone of the modern theory of (projective) K\"ahler manifolds. Although it is known that the holomorphic sectional curvature completely determines the curvature tensor, there is no direct local link between its sign and that of the Ricci curvature. In this paper, we provide results in this direction for projective K\"ahler manifolds of semi-negative (i.e., non-positive) holomorphic sectional curvature by analyzing the structural implications of the curvature assumption with the help of a new invariant that records the largest codimension of maximal subspaces in the tangent spaces on which the holomorphic sectional curvature vanishes.\par

In our previous paper \cite{heier_lu_wong_mrl}, we investigated the implications of negative holomorphic sectional curvature for the positivity of the canonical line bundle $K_M$ on a projective K\"ahler manifold $M$. Our main results in that paper can be summed up in the following theorem. In Section \ref{sec_def} below, the reader will find the definitions of the notions involved. We shall always work over the field of complex numbers.
\begin{theorem}[\cite{heier_lu_wong_mrl}]\label{mrl_theorem}
Let $M$ be a projective manifold with a K\"ahler metric of negative holomorphic sectional curvature. Then
\begin{enumerate}
\item the numerical dimension of $M$ is positive, and
\item the nef dimension of $M$ is equal to the dimension of $M$.
\end{enumerate}
\end{theorem}
It follows from a generalized Schwarz Lemma due to Ahlfors that on a compact hermitian manifold $M$ with negative holomorphic sectional curvature there exists no non-constant holomorphic map from the complex plane into $M$ (i.e., $M$ is Brody hyperbolic). In particular, there exist no rational curves on $M$. It thus follows from Mori's bend and break technique that a projective manifold with a K\"ahler metric of negative holomorphic sectional curvature has nef canonical line bundle. \par
Furthermore, it is a consequence of the Abundance Conjecture that for a projective manifold with nef canonical bundle, the Kodaira dimension equals the nef dimension (see the discussion at the beginning of Section \ref{abund_sec}). Thus, under the assumption of the Abundance Conjecture, Theorem \ref{mrl_theorem} implies that $M$ is of general type, i.e., its canonical bundle is big. Additionally, due to \cite{Kawamata_85}, a Brody hyperbolic projective manifold (or even one that is merely free of rational curves) has ample canonical  bundle if it is of general type. Thus, up to the validity of the Abundance Conjecture, our work in \cite{heier_lu_wong_mrl} proves the following conjecture, which the third named author learnt from S.-T.~Yau in personal conversations in the early 1970s.
\begin{conjecture}\label{yau_conj}
Let $M$ be a projective manifold with a K\"ahler metric of negative holomorphic sectional curvature. Then its canonical line bundle $K_M$ is ample.
\end{conjecture}
Since the Abundance Conjecture in dimension three is known by the works of Miyaoka and Kawamata (see \cite[Lecture IV]{Miyaoka_Peternell_book} for a nice account), our previous work in particular establishes the three dimensional case of Conjecture \ref{yau_conj} (\cite[Theorem 1.1]{heier_lu_wong_mrl}).\par
After the publication of our paper \cite{heier_lu_wong_mrl}, another paper on this topic appeared, namely \cite{wong_wu_yau}. Its main result is as follows.
\begin{theorem}[\cite{wong_wu_yau}]\label{wong_wu_yau_thm}
Let $M$ be a projective manifold of Picard number one. If $M$ admits a K\"ahler metric whose holomorphic sectional curvature is semi-negative everywhere and strictly negative at some point of $M$, then the canonical line  bundle of $M$ is ample.
\end{theorem}
The proof of this theorem as given in \cite{wong_wu_yau} is based on a refined Schwarz Lemma. The purpose of the present work is to treat the case of semi-negative holomorphic sectional curvature more comprehensively and in line with our earlier approach, but with
an integrated form of the Schwarz Lemma (Proposition \ref{incompatibility_thm}, see also Proposition \ref{prop_tot_geod}). In particular, we recover the above Theorem \ref{wong_wu_yau_thm}. \par

To state our first result, we make the following definitions.  For $p\in M$, let $\eta(p)$ be the maximum of those integers $k\in \{0,\ldots,n:=\dim M\}$ such that there exists a $k$-dimensional subspace $L\subset T_p M$ with $H(v)=0$ for all $v\in L\backslash \{\vec{0}\}$. Set $\eta_M:=\min_{p\in M} \eta(p)$ and $r_M:=n-\eta_M.$ Note that by definition $r_M=0$ if and only if $H$ vanishes identically. Also, $r_M=\dim M$ if and only there exists at least one point $p\in M$ such that $H$ is strictly negative at $p$. Moreover, $\eta(p)$ is upper-semicontinuous as a function of $p$, and consequently the set
$$\{p\in M\ |\ \eta(p)=\eta_M\}
$$
is an open set in $M$ (in the classical topology). Now, the first of our results is the following.
\begin{theorem}\label{mthm}
Let $M$ be a projective manifold with a K\"ahler metric of semi-negative holomorphic sectional curvature. Then $M$ contains no rational curves and the canonical line bundle $K_M$ is nef.  Moreover, if the holomorphic sectional curvature vanishes identically, then $M$ is an abelian variety up to a finite unramified covering. If the holomorphic sectional curvature does not vanish identically, then 
\begin{enumerate}
\item the numerical dimension of $M$ is strictly positive, and
\item the nef dimension of $M$ is greater than or equal to $r_M\geq 1$.
\end{enumerate}
\end{theorem}
For the proof, we follow the basic strategy used in \cite{heier_lu_wong_mrl}. We recall that 
\cite{heier_lu_wong_mrl} was partially inspired by an
earlier work of Peternell \cite{peternell} on Calabi-Yau and hyperbolic manifolds. Note that Theorem \ref{wong_wu_yau_thm} is an immediate corollary of Theorem \ref{mthm}(i) due the following simple lemma, applied to the case $L=K_M$.

\begin{lemma}
Let $M$ be a projective manifold of Picard number one. Let $L$ be a nef line bundle on $M$ which is of positive numerical dimension. Then $L$ is ample.
\end{lemma}
\begin{proof}
Let $A$ be an ample divisor on $M$. Then $L$ is numerically equivalent to $cA$ for some rational number $c$. Since $L$ is nef, we have $c\geq 0$. If we had $c=0$, then $L$ would be numerically trivial, and thus its numerical dimension would be equal to zero (see Remark \ref{num_dim_zero_rkm}) in violation of the assumption. So we have $c>0$, and $L$ is ample by the Nakai-Moishezon-Kleiman ampleness criterion.
\end{proof}
\begin{remark}
It is of course an interesting question if it is ever possible to have a strict inequality in Theorem \ref{mthm}(ii). Even when $\dim M=2$ and $r_M=1$,  it is unclear whether $M$ can be a surface of nef dimension $2$, i.e., a surface of general type.  Intuitively, a surface with $r_M=1$ should be a properly elliptic surface (\cite[p. 189]{BPV}) at least in the case when the metric is analytic.
\end{remark}
In view of the Abundance Conjecture (see the discussion at the beginning of Section \ref{abund_sec}), Theorem \ref{mthm} suggests that $r_M$ should in fact be a lower bound for the Kodaira dimension of $M$, which we denote by $\kod(M)$. We offer the following two theorems that establish a partial solution to the problem of Abundance under our curvature assumption.

\begin{theorem}\label{hlw_thm1_2_gen_intro}
Let $M$ be an $n$-dimensional projective manifold of Albanese dimension $d>n-4$. Let $M$ possess a K\"ahler metric of semi-negative holomorphic sectional curvature. Then $$\kod(M) \geq r_M-(n-d-\max\{0,r_M-d\}).$$
In particular, if $M$ has maximal Albanese dimension (i.e., $d=n$),
then we have $$\kod(M)\geq r_M.$$
\end{theorem}

\begin{theorem}\label{mthm_kod_version_intro}
Let $M$ be an $n$-dimensional projective K\"ahler manifold  
of semi-negative holomorphic sectional curvature.
Suppose the Abundance Conjecture holds up to dimension $e$ (which is currently known for $e=3$). Suppose 
$\kod(M)\geq n-e$.  Then $$\kod(M) \geq r_M.$$
\end{theorem} 

We remark again that, by \cite{Kawamata_85},
 manifolds of maximal Kodaira dimension without
rational curves have ample canonical bundles. Hence
the above two theorems represent generalizations of 
the key theorems of \cite{heier_lu_wong_mrl}, see
Section~\ref{abund_sec}. \par

Theorem~\ref{mthm_kod_version_intro} follows immediately from applying the subsequent proposition to the Kodaira-Iitaka map of $M$. The proposition generalizes Proposition~\ref{incompatibility_thm}
and Lemma~\ref{Shwarz} and is a general integrated (and non-equidimensional) form of the Schwarz Lemma which should be of strong independent interest.

\begin{proposition} \label{prop_tot_geod} Let $M$ be a projective manifold with a K\"ahler metric of semi-negative holomorphic sectional curvature.
Let $N$ be a $k$-dimensional projective variety with at most canonical (or even klt) singularities having pseudo-effective anti-canonical $\QQ$-Cartier divisor $-K_N$. Let $f:N\dra M$ be a rational map
that is generically finite, i.e., $df$ has rank $k$ somewhere.
Then $K_N$ is numerically trivial, and $f$ is a holomorphic immersion that
induces a flat metric on the smooth locus $N_{sm}$ of $N$ and is totally geodesic along $N_{sm}$. In particular, if $N$ is smooth,
then $N$ admits an abelian variety as an unramified covering
and $f$ is a totally geodesic holomorphic immersion that
induces a flat metric on $N$.
\end{proposition}
In dimension two, we are able to obtain the precise structure theorem below. Note that by the base point freeness of 
pluricanonical systems in dimension two, we can and do take the 
Kodaira-Iitaka map to be the morphism given by an appropriate pluricanonical map here.
\begin{theorem}\label{structure_thm_surface}
Let $M$ be a smooth projective surface with a K\"ahler metric of semi-negative holomorphic sectional curvature. Then one of the following will hold true.
\begin{enumerate}
\item $\kod (M) = 0\!:$ $M$ is an abelian surface up to a finite unramified covering, i.e., $M$ is an abelian surface or a hyperelliptic surface.
\item $\kod(M)=1\!:$ The Kodaira-Iitaka map of $M$ is an elliptic fibration whose only singular fibers are multiple elliptic curves. The base space is a smooth orbifold curve with ample orbifold canonical divisor. Moreover, $M$ admits a product of smooth curves $C\times F$ as a finite unramified covering and the metric of $M$ pulled back to $C\times F$ is the product of a non-flat metric of semi-negative curvature on $C$ and a flat metric on $F$ up to the addition of  mixed terms each a product of (anti)holomorphic one forms, one
from $C$ and one from $F$ as in Proposition \ref{metric_dec}.
\item $\kod (M) = 2\!:$ The canonical line bundle of $M$ is ample.
\end{enumerate} 
\end{theorem}
The proof of this theorem is partially based on the following general metric decomposition result valid in arbitrary dimension.
 \begin{proposition}\label{metric_dec}
Let $M=Y\times F$, where $Y$ and $F$ are projective manifolds. Let $\pi$ and $p$ be the projections to $Y$ and $F$ respectively. Let $\omega$ be a K\"ahler form on $M$
whose restrictions to the fibers of $\pi$ yield K\"ahler-Einstein metrics on these fibers. Then these
restrictions are in fact the pullback of a  K\"ahler-Einstein form 
$\omega_F$ on $F$ and 
$$\omega-p^*\omega_F=\pi^*\omega_Y
+\sum_i (\pi^*\mu_i\wedge\overline{ p^*\nu_i}
+\overline{\pi^*\mu_i}\wedge p^*\nu_i)$$
for a K\"ahler form $\omega_Y$ on $Y$ and 
holomorphic one forms $\mu_i$ on $Y$
and $\nu_i$ on $F$. In particular, if $Y$ or $F$ has 
zero irregularity, then $\omega$ corresponds to 
the product of a K\"ahler-Einstein
metric on $F$ and a K\"ahler metric on $Y$.
\end{proposition}
\par
\begin{remark}
Under stronger curvature assumptions such as semi-negative sectional or bisectional curvature, there is a considerable amount of previous work that yields structural results stronger than ours. Moreover, our invariant $r_M$ is similar to the more standard {\it Ricci rank}, which comes with an associated {\it Ricci kernel foliation}. Our present results regarding holomorphic sectional curvature can be seen as complementary to those earlier results. We refer the reader to \cite{Zheng_surfaces_london}, \cite{Wu_Zheng_JDG}, \cite{Zheng_helv}, \cite{liu} for further details.
\end{remark}

The contents of the sections of this paper can be summarized as follows. In Section \ref{sec_def}, we recall the key definitions and establish basic properties, in particular the incompatibility statement Proposition \ref{incompatibility_thm}. In Section \ref{sec_proof}, we shall prove Theorem \ref{mthm}. In Section 4, we discuss implications of the Abundance Conjecture, which includes the statement of corollaries to Theorem \ref{mthm} in dimension no greater than three. In Section \ref{high_alb_dim}, we discuss the case of positive Albanese dimension, where our principal theorem is 
Theorem~\ref{hlw_thm1_2_gen_intro}
(repeated as Theorem~\ref{hlw_thm1_2_gen}).
In Section~\ref{high_kod}, we prove 
Theorem~\ref{mthm_kod_version_intro}
(repeated as Theorem~\ref{mthm_kod_version})
and the related Proposition \ref{prop_tot_geod} (repeated as Proposition \ref{key_proposition})  
of independent interest.   
In Section~\ref{structure}, we prove the above structural theorem on the decomposition of surfaces.

\begin{acknowledgement}
The first author would like to thank CRM/CIRGET and the D\'epartement de Math\'ematiques at the Universit\'e du Qu\'ebec \`a Montr\'eal for their hospitality during the preparation of this paper. The second author would like to thank NSERC for its financial support that allowed its write-up. He is also indebted to Hongnian Huang for discussions related to Proposition~\ref{metric_dec}. We thank Fangyang Zheng for pointing out an issue with an earlier definition of $r_M$ in a previous version of this paper.
\end{acknowledgement}

 \section{Basic definitions and properties}\label{sec_def}
Let $M$ be an $n$-dimensional manifold with local coordinates $z_1,\ldots,z_n$. Let
\begin{equation*}
g=\sum_{i,j=1}^n g_{i\bar j} dz_i\otimes d\bar{z}_j
\end{equation*}
be a hermitian metric on $M$. The components $R_{i\bar j k \bar l}$ of the curvature tensor $R$ associated with the metric connection are locally given by the formula
\begin{equation*}
R_{i\bar j k \bar l}=-\frac{\partial^2 g_{i\bar j}}{\partial z_k\partial \bar z _l}+\sum_{p,q=1}^n g^{p\bar q}\frac{\partial g_{i\bar p}}{\partial z_k}\frac{\partial g_{q\bar j}}{\partial \bar z_l}.
\end{equation*}
If $g$ is K\"ahler, the Ricci curvature takes a particularly nice form. In fact, we can define the \textit{Ricci
curvature form} to be
\begin{equation*}
\Ric=-\sqrt{-1}\partial\bar\partial \log\det(g_{i\bar{j}}).
\end{equation*}
By a result of Chern, the class of the form $\frac{1}{2\pi}\Ric$ is equal to $c_1(M)=c_1(-K_M)$, where $K_M$ is the canonical line bundle of $M$. \par
In Section \ref{high_kod}, we also use the symbol $\Ric$ to denote the curvature form of a hermitian metric on a line bundle.\par\par
The \textit{scalar curvature} $S$ of $g$ is defined to be the trace of 
$\Ric$ with respect to a unitary frame.\par
It follows from linear algebra and the
definition of scalar curvature that
$$
\Ric \wedge \omega^{n-1} = \tfrac {2}{ n} S\, \omega^n,
$$
where $\omega=\frac{\sqrt{-1}}{2}\sum_{i,j=1}^n
g_{i\bar{j}} \mathrm{d}z_i\wedge \mathrm{d}\bar{z}_{j}$ is
the \textit{(1,1)-form associated to $g$}. In situations where there are several spaces, metrics, and associated forms involved, the reader should assume that the unadorned symbols $g$ and $\omega$ pertain to $M$.
\par

If $\xi=\sum_{i=1}^n\xi_i \frac{\partial }{\partial z_i}$ is a non-zero complex tangent vector at $p\in M$, then the holomorphic sectional curvature $H(\xi)$ is given by
\begin{equation*}
H(\xi)=\left( 2 \sum_{i,j,k,l=1}^n R_{i\bar j k \bar l}(p)\xi_i\bar\xi_j\xi_k\bar \xi_l\right) / \left(\sum_{i,j,k,l=1}^ng_{i\bar j}g_{k\bar l} \xi_i\bar\xi_j\xi_k\bar \xi_l\right).
\end{equation*}
An important fact about holomorphic sectional curvature is the following. If $M'$ is a submanifold of $M$, then the holomorphic sectional curvature of $M'$ does not exceed that of $M$. To be precise, if $\xi$ is a non-zero tangent vector to $M'$, then
\begin{equation*}
H'(\xi)\leq H(\xi),
\end{equation*}
where $H'$ is the holomorphic sectional curvature associated to the metric on $M'$ induced by $g$. For a short proof of this inequality see \cite[Lemma 1]{Wu}. Basically, the inequality is an immediate consequence of the Gauss-Codazzi equation.\par
We have the following pointwise result due to Berger \cite{Berger} (see \cite{Hall_Murphy} for a recent new approach).
\begin{theorem}[\cite{Berger}]\label{berger_theorem}
Let $M$ be a compact manifold with a K\"ahler metric of semi-negative holomorphic sectional curvature. Then the scalar curvature function 
$S$ is also semi-negative everywhere on $M$. Moreover, let $p\in M$ and assume that there exists $w\in T_pM\backslash \{\vec{0}\}$ such that $H(w) < 0$. Then $S(p) < 0$.
\end{theorem}
Berger's theorem is proven using a pointwise formula expressing the scalar curvature at a point in terms of the average holomorphic sectional curvature on the unit sphere in the tangent space at that point. Based on Berger's theorem, we have the following proposition.
\begin{proposition}\label{incompatibility_thm}
Let $M$ be a projective manifold whose first real Chern class is zero. Let $g$ be a K\"ahler metric on $M$ whose holomorphic sectional curvature is semi-negative. Then the holomorphic sectional curvature of $g$ vanishes identically and $M$ is an abelian variety up to a finite unramified covering.
\end{proposition}
\begin{proof}
Assume the holomorphic sectional curvature of $g$ does not vanish identically. Then there exists a point $p\in M$ and $w\in T_p M\backslash \{\vec{0}\}$ such that $H(w) < 0$. By Theorem \ref{berger_theorem}, the scalar curvature is non-positive everywhere, and $S(p) < 0$. Thus,
\begin{align}
0 & =2\pi \int _M c_1(-K_M) \wedge \omega^{n-1}\nonumber\\
&  =\int _M \Ric (g)\wedge \omega^{n-1}\nonumber\\
& =\int_M  \frac 2 n S\, \omega ^n <0 \nonumber,
\end{align}
which is a contradiction.\par
Having shown that the holomorphic sectional curvature of $g$ does vanish identically, it is immediate that $M$ is an abelian variety up to a finite unramified covering. Namely, it is a basic fact that the holomorphic sectional curvature of a K\"ahler metric completely determines the curvature tensor $R$ (\cite[Proposition 7.1, p. 166]{kobayashi_nomizu_ii}). In particular, if $H$ vanishes identically, then $R$ vanishes identically. However, due to \cite{Igusa}, a projective K\"ahler manifold with vanishing curvature tensor admits a finite unramified covering by an abelian variety.
\end{proof}
We conclude this section by defining the two notions of positivity of the canonical line bundle that appear in Theorem \ref{mthm}. Let $L$ be an arbitrary nef line bundle on $M$. Then the {\it numerical dimension of $L$}, which we denote $\nu(L),$ is $\max\{k\in\{0,1,\ldots,\dim M\}: (c_1^\RR(L))^k\not =  [0]\}$, where $c_1^\RR(L)$ denotes the first real Chern class of $L$. We write $\nu(M)$ for $\nu(K_M)$, the {\it numerical dimension of $M$} (aka the {\it numerical Kodaira dimension of $M$}). 
\begin{remark}\label{num_dim_zero_rkm}
It is immediate that $\nu(L)=0$, i.e., $c_1^\RR(L)=[0]$, implies that $L$ is numerically trivial. The converse also holds true, which is nicely explained in \cite[Remark 1.1.20]{Laz_I}.\par
\end{remark}
The notion of nef dimension is based on the following theorem (\cite{Tsuji}, \cite{8aut}).
\begin{theorem}
Let $L$ be a nef line bundle on a normal projective variety $M$. Then there exists an almost holomorphic dominant rational map $f:M\dashrightarrow Y$ with connected fibers, called a ``reduction map,'' such that
\begin{enumerate}
\item $L$ is numerically trivial on all compact fibers $F$ of $f$ with $\dim F = \dim M - \dim Y$, and
\item for every general point $x\in M$ and every irreducible curve $C$ passing through $x$ with $\dim f(C) > 0$, we have $L.C > 0$.
\end{enumerate}
The map $f$ is unique up to birational equivalence of $Y$.
\end{theorem}
We call $\dim Y$ the {\it nef dimension of $L$}. When we apply the above theorem with $L=K_M$, we call $n(M):=\dim Y$ the {\it nef dimension of $M$}. 

\section{Proof of Theorem \ref{mthm}}\label{sec_proof}
\subsection{The non-existence of rational curves and the nefness of the canonical line bundle}\label{nefness}
If $K_M$ is not nef, then, by the bend and break technique of Mori, $M$ contains a rational curve, i.e., there exists a non-constant holomorphic map $\PP^1 \to M$. Thus, it only remains to show the non-existence of rational curves, which has been known at least since the 1970's (e.g., via the disk condition in Shiffman's \cite{Shiffman_ext} or by \cite[Corollary 2]{Royden_80}). We state the absence of rational curves in the following lemma, which actually is a slightly stronger result. Note in particular that we require the hermitian manifold $(M,h)$ to be neither complete nor K\"ahler. The first proof is of an analytic nature and perhaps preferable to some readers, since it avoids the use of saturations of subsheaves. The second proof relies on an integrated version of the Schwarz Lemma and thus is very much in line with the general theme of this paper.
\begin{lemma}\label{Shwarz}
Let $N$ be a compact Riemann surface of genus $\gamma$ and  $(M,h)$ a (not necessarily complete) hermitian manifold of semi-negative holomorphic sectional curvature. If $f:N\to M$ is a non-constant
holomorphic map, then $\gamma\geq 1$ and $\gamma=1$ if and only if $f$ is
a totally geodesic immersion inducing a flat metric on $N$.
\end{lemma}

\begin{proof}[Proof of Lemma \ref{Shwarz} using analytic methods]
We write $\omega$ for the $(1,1)$-form of the hermitian manifold $(M,h)$ and $\omega_{KE}$ for the $(1,1)$-form of the 
K\"ahler-Einstein metric on $N$ with constant scalar curvature
$S=2-2\gamma$. Then
$$u:=\frac{f^*\omega}{\omega_{KE}}$$
is a non-negative smooth function on $N$ that is not identically
vanishing (otherwise $df$ would be identically zero and thus 
$f$ would be constant). Consequently, $df$ has at most a finite number
of zeros. Outside of these zeros,  $u$ is positive, and we have 
$$\sqrt{-1}\partial\overline \partial \log u = S\omega_{KE} - k f^*\omega, $$
where $k$ is the holomorphic sectional curvature of the K\"ahler
form $f^*\omega$ induced by $h$. Due to the curvature decreasing property of subbundles and submanifolds, we have $k\leq 0$. We
may assume that $S\geq 0$ since the theorem is vacuous otherwise.
We see then that $\log u$ is a subharmonic function and hence is
constant since $N$ is compact. This means that $u$ is a positive
constant function and hence that $df$ has no zeros. It also implies
that $S-ku=0$, forcing both $S=0$ and the everywhere vanishing
of $k$. We have thus shown that $\gamma= 1$ and that $f$ is a totally geodesic immersion
inducing a flat metric $f^*\omega$ on $N$. Since the converse of the last statement of the 
lemma is clear, the proof is now complete.
\end{proof}

\begin{proof}[Proof of Lemma \ref{Shwarz} using an integrated version of the Schwarz Lemma] We abuse notation and denote 
(holomorphic) vector bundles
and their sheaves of (holomorphic) sections by the same symbols.
Let $L$ be the saturation of the rank one subsheaf of $f^{-1}TM$
over $N$ given by the image of $TN$ via the differential of $f$ 
naturally given as section $df$ of $Hom_{N}(TN,f^{-1}TM)$.
Then $L$ is a subbundle of $f^{-1}TM$ with an induced 
hermitian metric $h_L$ and $df$ identifies with a section
$s$ of the line bundle $Hom_N(TN,L)$. Since $L$ is holomorphically
identified with $TN$ via $s$ over the dense open subset $N_0$
of $N$
where $s\not=0$ and since, over $N_0$, $h_L=f^*h$ is the induced metric 
on $T{N_0}$, the metric $h_L$ has semi-negative curvature on $N$ by the curvature decreasing property. Hence, the result follows from
$$2-2\gamma=\deg TN\leq \deg TN + \deg (s) =\deg L
=\int_N c_1(L, h_L)\leq 0$$
(which forces equality in the case $\gamma=1$) and from the 
Gauss-Codazzi equation.
\end{proof}

\subsection{The case of vanishing holomorphic sectional curvature}\label{flat_case_section}  If the holomorphic sectional curvature of $M$ vanishes, then the argument at the end of the proof of Proposition \ref{incompatibility_thm} shows that $M$ must be an abelian variety up to a finite unramified covering.

\subsection{Positivity of the numerical dimension when H does not vanish identically}
Next, we show that the numerical dimension of $M$ is positive when H does not vanish identically. To this end, let us assume that the numerical dimension of $M$ is zero. By definition, this means that the first real Chern class of $F$ is trivial. By Proposition \ref{incompatibility_thm}, this implies that the holomorphic sectional curvature vanishes identically, which is a contradiction to our assumption.\par

\subsection{The bound on the nef dimension}\label{nefd}
Now, we prove that the nef dimension $n(M)$ is greater than or equal to $r_M$. To this end, let us denote by $f:M\dashrightarrow Y$ a nef reduction map with respect to $K_M$, and let $I\subset M$ denote the set of points of indeterminacy of $f$. We write $f_h$ for the holomorphic map $f|_{M\backslash I}$.\par
Now, let us assume that $n(M) < r_M$ and derive a contradiction. Since $M$ is smooth, we can apply the generic smoothness theorem \cite[Corollary III.10.7]{Hartshorne} and conclude that there exists an open and dense subset $V\subset Y$ such that $f_h:f_h^{-1}(V)\to V$ is a smooth submersion. Since a smooth submersion is an open map, the set 
$$\tilde V := f_h(\{p\in M\ |\ \eta(p)=\eta_M \}\backslash I)$$ 
is a non-empty open subset (in the classical topology) of $V$. We pick a point $y\in \tilde V$ such that the fiber $F$ of $f_h$ over $y$ is compact and has the expected dimension $\delta:= n-n(M) > n-r_M$. \par
By the adjunction formula, $K_M|_F=K_F$, so $K_F$ is numerically trivial. By \cite[Remark 1.1.20]{Laz_I}, this implies that $c_1^\RR(K_F) = [0]$. Now, let $p\in F$ be such that $\eta(p)=\eta_M$. Due to $\delta > n - r_M =\eta_M$, there exists a nonzero vector $w$ in $T_pF \subset T_pM$ with $H(w)< 0$. Due to the curvature decreasing property mentioned in Section \ref{sec_def}, the holomorphic sectional curvature $H'$ of the K\"ahler metric $g'$ induced on $F$ by $g$ satisfies $H'(w) < 0$. Thus, $H'$ is semi-negative and does not vanish identically. By Proposition \ref{incompatibility_thm}, we have obtained a contradiction.

\section{Remarks related to the Abundance Conjecture} \label{abund_sec}
On a projective manifold $M$ with nef canonical line bundle $K_M$, the following chain of inequalities holds:
\begin{equation} \kod(M)\leq \nu(M) \leq n(M).\label{chain_of_inequ}\end{equation}
The first inequality was established by Kawamata in \cite[Proposition 2.2]{Kawamata_85_pluri_sys} and the second inequality is in \cite[Proposition 2.8]{8aut}. The name Abundance Conjecture is commonly used to refer to the claim that the first inequality is actually an equality, i.e., that Kodaira dimension and numerical dimension of $M$ agree (\cite[Conjecture 7.2]{Kawamata_85_pluri_sys}). \par
Furthermore, it is known that the Abundance Conjecture actually implies $\kod(M)= n(M)$, making \eqref{chain_of_inequ} an all-around equality. The argument for this goes as follows.\par
 If $n(M)=0$, then $0\leq \nu(M)\leq n(M)=0$ implies $\nu(M) =0$. Due to the Abundance Conjecture (which is actually \cite[Theorem 8.2]{Kawamata_85} in this case), $\kod(M)=\nu(M)=0=n(M)$.\par
If $n(M)>0$, then we observe first of all that $\nu(M) > 0$ also (by the definitions). The Abundance Conjecture now implies $\kod(M)=\nu(M) > 0$. For our purposes, it is convenient to think of the Kodaira-Iitaka map as represented by the map furnished by a large enough multiple of $K_M$ as described in \cite[Theorem 2.1.33]{Laz_I}. Due to the semi-ampleness of $K_M$ established in \cite[Theorem 1.1]{Kawamata_85_pluri_sys}, this map is holomorphic. A generic fiber $F$ has $\kod(F)=0$ and thus, again by the Abundance Conjecture, $\nu(F)=0$. Therefore, by construction of the nef reduction map, $\kod (M) \geq n(M)$. Together with \eqref{chain_of_inequ}, we obtain the desired equality. In fact, what we have seen is that the map furnished by a large enough multiple of $K_M$ can serve as a representative of both the nef reduction map with respect to $K_M$ and the Kodaira-Iitaka map (both of which are only defined up to birational equivalence) if the Abundance Conjecture is valid.\par
Due to the above, an immediate corollary to Theorem \ref{mthm} is the following.
\begin{corollary}\label{3d_cor_1}
Let $M$ be a projective manifold with a K\"ahler metric of semi-negative holomorphic sectional curvature. If $M$ satisfies the Abundance Conjecture, which is 
the case if its dimension is no greater than three,
then the Kodaira dimension of $M$ satisfies
$$\kod(M) \geq r_M.$$
\end{corollary}
In particular, a projective manifold $M$ of dimension $n\leq 3$ with a K\"ahler metric of semi-negative holomorphic sectional curvature satisfying $r_M=n$ is of general type. Moreover, as we saw in Section \ref{nefness}, $M$ contains no rational curves, so that based on \cite{Kawamata_85}, we obtain the following improvement of \cite[Theorem 1.1]{heier_lu_wong_mrl}.

\begin{corollary}\label{3d_cor_2}
Let $M$ be a projective manifold of dimension $n=r_M$ satisfying
the same hypotheses as in the above corollary. 
Then $K_M$ is ample.
\end{corollary}
In light of this, it is rather clear that Conjecture \ref{yau_conj} should be generalized as follows.
\begin{conjecture}\label{improv_conj}
Let $M$ be a projective manifold with a K\"ahler metric of semi-negative holomorphic sectional curvature. 
 Then the Kodaira dimension of $M$ satisfies $\kod(M) \geq r_M.$ In particular, if $r_M=\dim M$, then the canonical line bundle $K_M$ is ample.
\end{conjecture}
We will see in the next section that if $M$ is of high enough Albanese dimension, then Conjecture \ref{improv_conj} holds true for $M$ (without any use of the Abundance Conjecture). 
In Section~\ref{high_kod}, we will prove the conjecture for the case
when $\kod(M)$ is high enough.\par
We conclude this section by mentioning that a proof of the Abundance Conjecture was announced in \cite{Siu_AC}, although complete details of this proof seem to be not yet available.

\section{Manifolds of positive Albanese dimension}\label{high_alb_dim}
For projective manifolds of positive Albanese dimension, we actually can  prove some versions of Conjecture \ref{improv_conj}. Our results are as follows.
\begin{theorem}\label{mthm_max_alb_dim_version}
Let $M$ be an $n$-dimensional projective manifold whose Albanese dimension is maximal, i.e., equal to  $n$. Let $M$ possess a K\"ahler metric of semi-negative holomorphic sectional curvature. Then $$\kod(M) \geq r_M.$$
\end{theorem} 

\begin{proof}
By the definition of $M$ being of maximal Albanese dimension, $M$ possesses a generically finite map to an abelian variety $A$, namely its Albanese map $a$. Thus, the Kodaira dimension $\kod(M) \geq 0$, since one can pull back a not identically zero holomorphic $n$-form from the abelian variety under $a$, which will yield a not identically zero holomorphic $n$-form on $M$. \par
To warm up (and avoid trivialities in the treatment of the general case) we first deal with the case when $\kod(M)=0$. By \cite[Theorem 1]{kawamata_81}, the Albanese map $a:M\to A$ is a fiber space. Since the Albanese dimension of $M$ is maximal, the map $a$ is generically finite. Thus, we have so far established that $a$ is a birational holomorphic map. Since the Albanese torus $A$ is smooth, the exceptional set of $a$ is covered by rational curves due to \cite{Abh}. Since there are no rational curves on $M$, $a$ is injective. As an injective and onto holomorphic map between manifolds, $a$ is an isomorphism and $M$ is an abelian variety. By Proposition \ref{incompatibility_thm}, $r_M=0$, and the theorem is proven in the case $\kod(M)=0$.\par
We now treat the general case $\kod(M)>0$. Let $\pi:M^*\to Y^*$ be a holomorphic version of the Kodaira-Iitaka map of $M$ and $\sigma: M^*\to M$ the pertaining modification of $M$. We have a diagram 
\begin{equation*}\begin{CD}
M^* @>\pi>> Y^*\\
@ V \sigma  VV \\
M@>a>> A\\
\end{CD}.
\end{equation*}\par 
\ \par
Let $G$ be a general fiber of $\pi$. Such a $G$ is a submanifold of $M^*$ of dimension $n-\kod(M)$ with $\kod(G)=0$. The map $a\circ\sigma |_G:G \to (a\circ\sigma)(G)$ is a generically finite holomorphic map. We let $B:=(a\circ\sigma)(G)$. Due to \cite[Corollary 9]{kawamata_81}, $0=\kod(G) \geq \kod(B)$. Moreover, \cite[Lemma 10.1]{Ueno} yields $\kod(B) \geq 0$, so $\kod(B) = 0$. By \cite[Theorem 10.3]{Ueno}, $B$ is the translate of an abelian subvariety $A_0$ of $A$. Since there are only countably many abelian subvarieties of $A$, we can assume that $A_0$ does not depend on $G$. We write $p:A\to A/A_0$ for the canonical projection.\par
Next, we observe that the map $\sigma |_G:G \to \sigma(G)$ is a birational holomorphic map. Now, note that $A/A_0$ is again an abelian variety and consider the map $p\circ a : M \to A/A_0$. An irreducible component of a general fiber of $p\circ a$ will be of the form $\sigma(G)$, where $G$ is a general fiber of $\pi$. Due to generic smoothness, $\sigma(G)$ will be smooth, and due to birational invariance, $\kod (\sigma(G)) = 0$. By definition of the Kodaira-Iitaka map, the dimension of $\sigma(G)$ is $n-\kod(M)$. In fact, $\sigma(G)$ is an abelian variety for the following reason. In our situation, it follows from \cite[Main Theorem]{Kawamata_Viehweg_char_ab_var} that $\sigma(G)$ is birational to an abelian variety. Since its image $(a\circ\sigma)(G)$ is smooth, the same argument as above based on \cite{Abh} shows that $\sigma(G)$ is isomorphic to an abelian variety.\par
To conclude the proof, assume that the general fiber $G$ is chosen such that the abelian variety $\sigma(G)$ has non-empty intersection with the open set $\{p\in M\ |\ \eta(p)=\eta_M\}$. By the curvature decreasing property and Proposition \ref{incompatibility_thm}, $\dim \sigma(G) =n-\kod(M)$ can be no greater than $\eta_M$, i.e., $\kod(M) \geq n-\eta_M=r_M$.\end{proof}
The above Theorem \ref{mthm_max_alb_dim_version} states in particular that a projective manifold $M$ of maximal Albanese dimension with a K\"ahler metric of semi-negative holomorphic sectional curvature satisfying $r_M=\dim M$ is of general type. The established absence of rational curves, together with \cite{Kawamata_85}, then yields the following special case of Conjecture \ref{improv_conj}.
\begin{corollary}\label{r_max_and_alb_max}
Let $M$ be a projective manifold with a K\"ahler metric of semi-negative holomorphic sectional curvature. Assume that $M$ is of maximal Albanese dimension and that $r_M=\dim M$. Then the canonical line bundle of $M$ is ample.
\end{corollary}
The following theorem represents a generalization of \cite[Theorem 1.2]{heier_lu_wong_mrl}. It contains Theorem \ref{mthm_max_alb_dim_version} as the special case $d=n$, but since the bound below may be somewhat hard to parse on a first reading and since fewer deep results had to be cited in the earlier proof, we thought it best to isolate the more concise Theorem \ref{mthm_max_alb_dim_version} at the beginning of this section.

\begin{theorem}\label{hlw_thm1_2_gen}
Let $M$ be an $n$-dimensional projective manifold of Albanese dimension $d>n-4$. Let $M$ possess a K\"ahler metric of semi-negative holomorphic sectional curvature. Then $$\kod(M) \geq r_M-(n-d-\max\{0,r_M-d\}).$$
\end{theorem}

\begin{proof}
Under the assumption $d > n - 4$, the Albanese map $a:M\to a(M)$ is a holomorphic map such that an irreducible component $F$ of a general fiber has dimension $n-d \leq 3$. Since the Iitaka Conjecture holds in the case of fibers of dimension no greater than three (\cite{Kawamata_85}, see also \cite{Birkar} for some expository comments), we have $\kod(M) \geq 0$. Moreover, by Corollary \ref{3d_cor_1} and the curvature decreasing property, $\kod(F)\geq r_F$. Again by the curvature decreasing property, $r_F \geq  \max\{0,r_M-d\}$. \par
We  again consider the diagram as in the proof of Theorem \ref{mthm_max_alb_dim_version}. We denote by $\tilde F$ the strict transform of $F$ under $\sigma$, which satisfies $\kod(\tilde F)=\kod(F)\geq r_F \geq \max\{0,r_M-d\} $.\par
The restriction of $\pi$ to $\tilde F$ gives a holomorphic map $\pi: \tilde F\to\pi(\tilde F)$. If we denote an irreducible component of a general fiber of this map by $S$, then the Easy Addition Formula (applied after a Stein factorization) yields
\begin{equation}\label{easy_add_form}
\max\{0,r_M-d\} \leq r_F \leq \kod ( \tilde F) \leq \kod (S)+\dim \pi( \tilde F).
\end{equation}
Moreover, it is clear that $n-d=\dim \tilde F=\dim S +\dim \pi( \tilde F)$.\par
On the other hand, let $G$ be the fiber of $\pi:M^*\to Y^*$ such that $S$ is a component of $G\cap\tilde F$. By the standard properties of the Kodaira-Iitaka map, when $F$ and $S$ are appropriately chosen, $G$ is a projective manifold with $\kod(G)=0$. Due to the resolved Iitaka Conjecture in the case of fibers of dimension no greater than three,
\begin{equation*}
0=\kod(G)\geq \kod (S)+\kod(a(\sigma(G))).
\end{equation*}
By \cite[Lemma 10.1]{Ueno}, we know that $\kod((a\circ\sigma)(G))$ is at least $0$. From \eqref{easy_add_form}, it is also clear that $\kod (S) \geq 0$. Hence, $0=\kod (S)=\kod((a\circ \sigma)(G))$, and by \cite[Theorem 10.3]{Ueno}, $(a\circ \sigma)(G)$ is the translate of an abelian subvariety. Again from \eqref{easy_add_form}, we infer $\dim \pi( \tilde F) \geq\max\{0,r_M-d\} $ and thus $\dim (S) \leq n-d-\max\{0,r_M-d\}$.\par
Next, observe that $\dim (a(\sigma(G)))$ is bounded below by
\begin{align*} 
\dim (G) -\dim (S) &\geq \dim(G) -(n-d-\max\{0,r_M-d\} ) \\ &= n-\kod(M) -(n-d-\max\{0,r_M-d\}).
\end{align*}
We have now established that $a(\sigma(G))$ is the translate of an abelian subvariety $A_0$ of dimension at least $n-\kod(M) -(n-d-\max\{0,r_M-d\})$. We write $p:A\to A/A_0$ for the canonical projection. As in the proof of Theorem \ref{mthm_max_alb_dim_version}, this yields
$$n-\kod(M) -(n-d-\max\{0,r_M-d\})\leq n-r_M,$$ 
i.e., $\kod(M) \geq r_M-(n-d-\max\{0,r_M-d\})$.
\end{proof}
Note that if $n=r_M$, then $r_M-(n-d-\max\{0,r_M-d\})=r_M$ in the statement of Theorem \ref{hlw_thm1_2_gen}. Therefore, Corollary \ref{r_max_and_alb_max} can be strengthened to the following.
\begin{corollary}\label{hlw_thm1_2_gen_with_rm_max}
Let $M$ be a projective manifold with a K\"ahler metric of semi-negative holomorphic sectional curvature. Assume that $r_M=\dim M$ and that the Albanese dimension of $M$ is greater than $\dim M - 4$. Then the canonical line bundle of $M$ is ample. \end{corollary} 
The following final theorem of this section applies in the case of arbitrary positive Albanese dimension. The assumption of non-negative Kodaira dimension is necessary, because based on the other assumptions, we cannot prove the existence of any pluricanonical sections. Without at least one pluricanonical section, there is no Kodaira-Iitaka map, and our argument does not work.  
\begin{theorem}\label{thm_any_alb_dim_with_iitaka}
Let $M$ be an $n$-dimensional projective manifold of Kodaira dimension at least zero and Albanese dimension $d > 0$. Let $M$ possess a K\"ahler metric of semi-negative holomorphic sectional curvature. Contingent on the validity of the Iitaka Conjecture for fibrations with fiber dimension no greater than $n-d$, the following holds: $$\kod(M) \geq r_M-(n-d).$$
\end{theorem}
\begin{proof}
The proof is a simplified version of the proof of Theorem \ref{hlw_thm1_2_gen}. We let $F$, $\tilde F$, $S$, $G$ be as above. In the present situation, we cannot rule out $\kod(\tilde F)=-\infty$ (although it is ruled out conjecturally by Conjecture \ref{improv_conj}), so the Easy Addition Formula \eqref{easy_add_form} becomes vacuous. Instead, we apply the Easy Addition Formula to $a:\sigma(G)\to (a\circ\sigma)(G)$, which yields
\begin{equation*}
0=\kod(G)\leq \kod (S)+\dim(a(\sigma(G))).
\end{equation*}
Thus, $ \kod (S) \geq 0$. Moreover, due to the Iitaka Conjecture in the case of fibers of dimension no greater than $n-d$, we have
\begin{equation*}
0=\kod(G)\geq \kod (S)+\kod(a(\sigma(G))).
\end{equation*}
Again, we conclude $\kod(a(\sigma(G)))=0$. It remains to observe that $\dim (a(\sigma(G)))$ is bounded below by $\dim G -\dim S \geq \dim G - \dim F = n-\kod(M) -(n-d).$
As before, we conclude 
$$n-\kod(M) -(n-d)\leq n-r_M,$$ 
i.e., $\kod(M) \geq r_M-(n-d)$.
\end{proof}

\section{Manifolds of high Kodaira dimension}\label{high_kod}
If we assume the validity of the Abundance Conjecture up to some dimension $e$ with $1\leq e \leq n$, then we can prove the desired inequality $\kod(M) \geq r_M$ provided we may additionally assume that $\kod(M) \geq n-e$ (which is of course a non-vacuous statement only if $r_M > n-e$).

\begin{theorem}\label{mthm_kod_version}
Let $M$ be an $n$-dimensional projective K\"ahler manifold  
of semi-negative holomorphic sectional curvature.
Suppose the Abundance Conjecture holds up to dimension $e$ (which is currently known for $e=3$). Suppose 
$\kod(M)\geq n-e$.  Then $$\kod(M) \geq r_M.$$
\end{theorem} 

The theorem follows readily from the following proposition
applied to the general fibers of the Kodaira-Iitaka map of $M$.

\begin{proposition} \label{key_proposition} Let $M$ be a projective manifold with a K\"ahler metric of semi-negative holomorphic sectional curvature.
Let $N$ be a $k$-dimensional projective variety with at most canonical (or even klt) singularities having pseudo-effective anti-canonical $\QQ$-Cartier divisor $-K_N$. Let $f:N\dra M$ be a rational map
that is generically finite, i.e., $df$ has rank $k$ somewhere.
Then $K_N$ is numerically trivial, and $f$ is a holomorphic immersion that
induces a flat metric on the smooth locus $N_{sm}$ of $N$ and is totally geodesic along $N_{sm}$. In particular, if $N$ is smooth,
then $N$ admits an abelian variety as an unramified covering
and $f$ is a totally geodesic holomorphic immersion that
induces a flat metric on $N$.
\end{proposition}

\begin{proof}
We will abuse notation and denote Cartier divisors and their
associated invertible sheaves as well as bundles and their
sheaves of sections with the same symbols, respectively. All metrics on complex bundles are understood to be hermitian. For simplicity, we will not distinguish a metric from its associated $(1,1)$-form.\par

Since $M$ has no rational curves, $f$ is in fact a holomorphic
map (see, for example, \cite{KM}). By the hypothesis on $f$, it has rank $k$ on a dense Zariski open set of $N_{sm}$. Hence, $\det(df)$ gives rise to a 
nontrivial section $s_{N_0}$ of the locally free sheaf
$Hom_{N_0}(K_{N_0}^\vee, f_0^{-1}(\Lambda^k TM))$ 
on every Zariski open subset $N_0$ of $N_{sm}$, where $f_0=f|_{N_0}$. \par

Let $\tau:X\to N$ be a resolution of 
the singularities of $N$ and $F=f\circ \tau: X\to M$. To avoid heavy notation, on the open subset $N_0\subset N_{sm}$ where the birational
map $F^{-1}$ is holomorphic, we identify $f|_{N_0}$ with $F|_{F^{-1}(N_0)}$. As before, $\det(dF)$ gives rise to a 
nontrivial section $s_X$ over $X$ of 
$Hom_X(K_{X}^\vee, F^{-1}(\Lambda^k TM))$
and $s_X|_{N_0}=s_{N_0}$ under the identification of $N_0$ with $F^{-1}(N_0)$. Here, the symbol $F^{-1}(\Lambda^k TM)$ simply denotes the pull-back of the vector bundle $\Lambda^k TM$. \par
Since $K_X$
is invertible, after replacing $X$ by some further blowup of $X$
if necessary, the same proof as in the resolution of the base locus
of a linear system into only divisorial parts 
(for example by blowing up the ideal sheaf given by
the image by $s_X$ of the vector sheaf 
$Hom_X(K_{X}^\vee, F^{-1}(\Lambda^k TM))^\vee$
in ${\OO}_X$)
allows us to assume
that the subscheme defined by $s_X=0$ is of pure codimension one,
i.e., a divisor $D$. This means that the saturation of the 
subsheaf $s_X(K_X^\vee)$ is given by a line subbundle $L$
of $F^{-1}(\Lambda^k TM)$ and $s_X$ can be identified with a section 
$s$ of the line bundle $K_X\otimes L=Hom_X(K_X^\vee, L)$ 
over $X$. 
Clearly $(s)=D$ on $X$ by 
construction and $s|_{N_0}=s_{N_0}$. Note that 
$$K_X=\tau^*K_N+E$$ for an effective $\QQ$-divisor $E$ supported
on the exceptional locus of $\tau$ and that both $E$ and $K_N$ are 
Cartier outside the singular locus of $N$.\par
Let $N_{00}\subset N_0$ be the Zariski open dense subset on which $df:TN\to TM$ has maximal rank, i.e., on which $f$ is an immersion. Theorem~\ref{berger_theorem} applied to the
induced metric $\omega_{00}=f_{00}^*\omega$ on 
$N_{00}$ where $f_{00}=f|_{N_{00}}$  together with the Gauss-Codazzi equation shows that the scalar
curvature $S_{\omega_{00}}$ of $\omega_{00}$ is semi-negative on $N_{00}$. Moreover, it vanishes identically there if and only if 
$f_{00}:(N_{00},\omega_{00})\to (M, \omega)$
is totally geodesic and $\omega_{00}$ is flat. 
We now proceed to show that not only the latter is the case but that
in fact $df$ has maximal rank over $N_{sm}$ so that
$f$ is a totally geodesic immersion with the induced 
flat metric there and in particular
$N_0=N_{00}$.

Since $L$ is a subbundle of $F^{-1}(\Lambda^k TM)$ and
$\Lambda^k\omega$ is a metric on $\Lambda^k TM$, we see 
that $L$ has an induced metric $h$ which restricts to $\det \omega_{00}$
on $N_{00}$ and thus $\Ric(h)=\Ric(\omega_{00})$ on $N_{00}$. 
Here, $\det \omega_{00}$ is a metric on $\det T{N_{00}}$ identified 
with $L|_{N_{00}}$ via $s$.
As the holomorphic
sectional curvature decreases on subvarieties and as
$\omega_{00}$ is K\"ahler on $N_{00}$,  $(N_{00},\omega_{00})$ 
has scalar curvature 
$S_{\omega_{00}}\leq 0$ by Berger's theorem. Consequently, we have 
$$\int_X c_1(L)\wedge F^*\omega^{k-1}=
\frac{1}{2\pi}\int_X\Ric(h)\wedge F^*\omega^{k-1}=
\frac{1}{n\pi}\int_{N_{00}}S_{\omega_{00}} \omega_{00}^k \leq 0\ ,$$
with equality in the inequality if and only if $\omega_{00}$
is flat and $f_{00}$ totally geodesic. But the first
integral above is the sum of the following two integrals:
$$\int_X c_1(K_X\otimes L)\wedge F^*\omega^{k-1}=
\int_X c_1(D)\wedge F^*\omega^{k-1}
=\int_D  i^*F^*\omega^{k-1},
$$
where $i$ is the inclusion of $D_{red}$ in $X$, and
(as $E$ is $\tau$-exceptional)
\begin{align*}\int_X c_1(-K_X)\wedge F^*\omega^{k-1}& =
\int_X c_1(-\tau^*K_N-E)\wedge \tau^*f^*\omega^{k-1}\\ &=
\int_N c_1(-K_N)\wedge f^*\omega^{k-1},
\end{align*}
both of which are semi-positive since $D$ is effective 
and $-K_N$ pseudo-effective. This forces all of the above
integrals to vanish. In particular, $\omega_{00}$ is flat on $N_{00}$
and therefore, since $\Ric(h)=\Ric(\omega_{00})=0$ on $N_{00}$,
$\Ric(h)=0$ on $X$. We then have, 
$-K_N$ being pseudo-effective, that for a generic curve $C$ cut
out by hyperplanes on $N$:
$$0\leq D.\tau^*C=K_N.C\leq 0,$$
forcing equality. Hence $D$ is $\tau$-exceptional and
$s$ is nowhere zero on $N_{00}$. By \cite[Ch. I, $\mathsection$ 4, Prop. 3]{Kleiman}, it also follows that $K_N$ is numerically trivial.\par
Finally, to see that $f$ is totally geodesic along $N_{sm}$, we argue as follows. Since the
exceptional divisor $E$ of $\tau$ is rationally connected, the condition $c_1(L)=0$ 
implies that $L$ is trivial on $E$ and that its inclusion into the 
trivial bundle $F^{-1}(\Lambda^k TM)|_{E}$ is constant. Hence
the subbundle $L$ of $F^{-1}(\Lambda^k TM)$ is the pullback
of a subbundle $\tilde L$ of $f^{-1}(\Lambda^k TM)$ on $N$. This
means that the inclusion of $K_X^\vee$ in $L$ 
over $\tau^{-1}(N_{sm})$ factors through the 
inclusion $\tilde s$ of $K_{N_{sm}}^\vee$ in $\tilde L|_{N_{sm}}$ 
given by the section 
$\det(df_{sm})$ of $Hom(K_{N_{sm}}^\vee, f_{sm}^{-1}(\Lambda^k TM))$, where $f_{sm}=f|_{N_{sm}}$. Since $\tilde s=s$ on $N_{00}$, where we have shown it is nowhere zero,
and since the complement of $N_{00}$ in $N_{sm}$
has codimension two or higher, 
the Cartier
divisor $(\tilde s|_{N_{sm}})$ must be zero on $N_{sm}$ and hence
$\det(df)$ is nowhere zero on $N_{sm}$. We may now conclude
that $f|_{N_{sm}}$ is a totally geodesic immersion as before.
\end{proof}

\begin{remark}
The above proposition generalizes Proposition~\ref{incompatibility_thm}
and Lemma~\ref{Shwarz}. Observe that the case of 
$\kod(M)\geq n-2$ and $r_M=n$ can be obtained directly 
just from Lemma~\ref{Shwarz} since Kodaira 
dimension zero minimal surfaces are dominated either
by an abelian surface or by families of elliptic curves by
\cite{MM}.
\end{remark}

\begin{remark}
Although we do not need it in this paper, with a little
further work, $N$ can be shown to be smooth. Also,
it is clear from the above proof that the projectivity assumption on $M$ 
is unnecessary and the singularity assumption on $N$ is made only to guarantee that $K_N$ is $\QQ$-Cartier if $f$ is a morphism. Implications for $M$ and its pluricanonical systems will be discussed elsewhere.
\end{remark}

\section{Structural decomposition theorem in dimension two}\label{structure}
In this section, we prove Theorem \ref{structure_thm_surface}. Under its assumptions, the nef fibration is known to be a morphism given by a pluricanonical map, which we take to be the Kodaira-Iitaka map. Out of the three possibilities for the Kodaira dimension, the case $\kod (M) = 1$ is the key one to treat. In this case, we denote the Kodaira-Iitaka map by $\pi:M\to Y$, where $Y$ is a smooth curve. Now, $\pi$ induces an orbifold structure on $Y$ by assigning the
multiplicity $m(q)$ to a point $q$ on $Y$ given by 
the multiplicity
of the generic fiber of $\pi$ above $q$. The orbifold $Y^\partial$
so endowed has the canonical $\QQ$-divisor 
$K_{Y^\partial}=K_Y+\sum_i (1-\frac{1}{m(q_i)})q_i$ following
the notation of \cite{Lu02}.
\begin{proof}[Proof of Theorem \ref{structure_thm_surface}]
It follows from Theorem \ref{mthm} and the validity of the Abundance Conjecture on surfaces that $\kod (M) \geq 0$. In case $\kod (M) = 0$, Theorem~\ref{mthm} states that $M$ has a finite unramified cover by an abelian surface, so there is nothing to prove. Also, in case $\kod (M) = 2$, we have already seen that $K_M$ is ample. So it remains to treat the case $\kod (M) = 1$.\par
Under the present assumptions, $\pi$ is an elliptic fibration whose only degenerate fibers are
multiples of elliptic curves, as any other type of degenerate fiber would contain rational curves (see \cite[p. 150, Table 3. Kodaira's table of singular elliptic fibers]{BPV}).
It follows that the canonical divisor $K_M$ is the pullback of the
orbifold canonical $\QQ$-divisor $K_{Y^\partial}$
of the base, see for example \cite{BL}. 
This means that $K_{Y^\partial}$ is ample and that ${Y^\partial}$ is
the quotient of the unit disk by a Fuchsian group.
As such a group has finite index torsion free subgroups 
by Selberg's lemma (see \cite{JR}), there is a finite 
cover of $Y$ by a smooth curve $C$ of genus at least two branched over $Y$ to precisely the
same multiplicity as that given by $\pi$. After the base change to $C$ and a normalization, we obtain a fibration over $C$ whose total space $\tilde M$ is an unramified covering of $M$ and a holomorphic fiber bundle over $C$ (the 
$j$-invariant of the fibers gives rise to a holomorphic map  
$j:C\to\CC$ and this forces $j$ to be constant). Replacing $C$
by a finite unramified covering if necessary, we then 
have $\tilde M=C\times F$
where $F$ is an elliptic curve and $\tilde M$ is a finite unramified
covering of $M$. The existence of such an unramified cover of $C$ is proven in \cite[Prop. VI.8]{Bea96} based on the existence 
and the affineness of the fine moduli spaces
for elliptic curves with level structure.\par

Proposition~\ref{metric_dec}
now shows that the pull-back metric $\tilde g$ on
$\tilde M=C\times F$ of $g$
is the product of a flat metric on $F$
with a metric $g_C$ on $C$ up to adding a term 
corresponding to a $(1,1)$-form of the form
$\sum_i (p_1^*\mu_i\wedge \overline{p_2^*\nu_i}
+\overline{p_1^*\mu_i}\wedge p_2^*\nu_i)$ where $p_1$, $p_2$ are
the projections and $\mu_i$, $\nu_i$ holomorphic one forms
on $C$ and $F$, respectively. But this additional term vanishes
if we pull back by a constant section of $p_1$ (a fiber of $p_2$) 
so that $g_C$ corresponds to the induced
metric by  $\tilde g$ on this fiber. By the curvature decreasing
properties on subvarieties, it follows that the holomorphic
sectional curvature of $g_C$ is semi-negative.
\end{proof}
Before we prove Proposition~\ref{metric_dec}, we note that its statement (and proof) bears resemblance to Zheng's theorem in \cite[p.~672]{zheng_annals}, which comes with an assumption of semi-negative {\it bi}\/{}sectional curvature. This assumption is not present in our case, but instead we can invoke the uniqueness of K\"ahler-Einstein metrics in a given K\"ahler class at a crucial point of the proof.
\begin{proof}[Proof of Proposition~\ref{metric_dec}]
Since $\omega$ is K\"ahler, the K\"unneth formula for 
$(1,1)$-cohomology classes and the global $\partial\bar\partial$-lemma
show that there exist a real $C^\infty$ function $\phi$ on $M$,
real $(1,1)$-forms $\omega_F$ on $F$ and $\omega_Y$ on $Y$
such that
$$\omega-p^*\omega_F
-\sqrt{-1}\partial\bar\partial \phi=\pi^*\omega_Y
+\sum_i (\pi^*\mu_i\wedge \overline{p^*\nu_i}
+\overline{\pi^*\mu_i}\wedge p^*\nu_i)$$
for holomorphic one forms $\nu_i$ on $F$ and $\mu_i$ on $Y$. Now, the right hand side pulls back to zero by
each constant section $s_y:F\to M$ since it factors 
through the inclusion $i_y:M_y\hookrightarrow M$
of the fiber $M_y=\pi^{-1}(y)$. It follows that 
$s_y^*\omega$ is cohomologous to $\omega_F$.
Since the former is Einstein, by the uniqueness 
of such $(1,1)$-forms in a given cohomology class
(\cite{Calabi}, \cite{Yau}) we have
$i_y^*\sqrt{-1}\partial\bar\partial \phi=0$ 
 for all $y\in Y$ if we choose
$\omega_F$ to be this unique K\"ahler form. 
With this choice, $\phi$ is a harmonic function 
on $M_y$ and therefore constant on $M_y$ for all $y\in Y$.
This means that $\phi$ is a function of $y$ only and thus
the term $\sqrt{-1}\partial\bar\partial \phi$ above may be
absorbed into $\omega_Y$. The $(1,1)$-form so formed on $Y$
is necessarily a K\"ahler form since its pullback to a
horizontal fiber (a fiber of $p$) is 
also the pullback of the K\"ahler form $\omega$.
\end{proof}

\begin{remark} A brief computation shows that if $\omega$
pulls back to flat metrics on the fibers of $\pi$, then these
fibers with the induced metric are totally geodesic. It should also be noted that the projectivity assumption in Proposition \ref{metric_dec} is merely made to make the proposition appear in line with the overall setting of the paper. Its proof does not use it and in fact works for K\"ahler manifolds.
\end{remark}

\begin{remark}
After this paper had been completed and gone to press at the Journal of Differential Geometry, Wu and Yau gave a proof of Conjecture \ref{yau_conj} in \cite{wu_yau}. Subsequently, Tosatti and Yang in \cite{tosatti_yang} extended that proof to the K\"ahler case. However, since these new proofs, which are based  on the complex Monge-Amp\`ere equation and a refined Schwarz Lemma, seem to require strict negativity of the holomorphic sectional curvature, the relationship of these works with our results regarding semi-negative holomorphic sectional curvature remains unclear. 
\end{remark}

\newcommand{\etalchar}[1]{$^{#1}$}


\begin{thebibliography}{BPVdV84}

\bibitem[Abh56]{Abh}
S.~Abhyankar.
\newblock On the valuations centered in a local domain.
\newblock {\em Amer. J. Math.}, 78:321--348, 1956. MR0082477, Zbl 0074.26301.

\bibitem[BCE{\etalchar{+}}02]{8aut}
Th. Bauer, F.~Campana, Th. Eckl, S.~Kebekus, Th. Peternell, S.~Rams,
  T.~Szemberg, and L.~Wotzlaw.
\newblock A reduction map for nef line bundles.
\newblock In {\em Complex geometry ({G}\"ottingen, 2000)}, pages 27--36.
  Springer, Berlin, 2002. MR1922095, Zbl 1054.14019.

\bibitem[Bea96]{Bea96}
A.~Beauville.
\newblock {\em Complex algebraic surfaces}, volume~34 of {\em London
  Mathematical Society Student Texts}.
\newblock Cambridge University Press, Cambridge, second edition, 1996.
\newblock Translated from the 1978 French original by R. Barlow, with
  assistance from N. I. Shepherd-Barron and M. Reid. MR1406314, Zbl 0849.14014.  

\bibitem[Ber66]{Berger}
M.~Berger.
\newblock Sur les vari\'et\'es d'{E}instein compactes.
\newblock In {\em Comptes {R}endus de la {III}e {R}\'eunion du {G}roupement des
  {M}ath\'ematiciens d'{E}xpression {L}atine ({N}amur, 1965)}, pages 35--55.
  Librairie Universitaire, Louvain, 1966. MR0238226, Zbl 0178.56001.

\bibitem[Bir09]{Birkar}
C.~Birkar.
\newblock The {I}itaka conjecture {$C_{n,m}$} in dimension six.
\newblock {\em Compos. Math.}, 145(6):1442--1446, 2009. MR2575089, Zbl 1186.14015.

\bibitem[BL00]{BL}
G.~Buzzard and S.~Lu.
\newblock Algebraic surfaces holomorphically dominable by {$\bold C^2$}.
\newblock {\em Invent. Math.}, 139(3):617--659, 2000. MR1738063, Zbl 0967.14025.

\bibitem[BPVdV84]{BPV}
W.~Barth, C.~Peters, and A.~Van~de Ven.
\newblock {\em Compact complex surfaces}, volume~4 of {\em Ergebnisse der
  Mathematik und ihrer Grenzgebiete (3) [Results in Mathematics and Related
  Areas (3)]}.
\newblock Springer-Verlag, Berlin, 1984. MR0749574, Zbl 0718.14023.

\bibitem[Cal57]{Calabi}
E.~Calabi.
\newblock On {K}\"ahler manifolds with vanishing canonical class.
\newblock In {\em Algebraic geometry and topology. {A} symposium in honor of
  {S}. {L}efschetz}, pages 78--89. Princeton University Press, Princeton, N.
  J., 1957. MR0085583, Zbl 0080.15002.

\bibitem[HM13]{Hall_Murphy}
S.~Hall, T.~Murphy.
\newblock Rigidity results for Hermitian-Einstein manifolds.
\newblock {\em {\rm arXiv:1311.6279}}, 2013.

\bibitem[Har77]{Hartshorne}
R.~Hartshorne.
\newblock {\em Algebraic geometry}.
\newblock Springer-Verlag, New York, 1977.
\newblock Graduate Texts in Mathematics, No. 52. MR0463157, Zbl 0367.14001.

\bibitem[HLW10]{heier_lu_wong_mrl}
G.~Heier, S.~Lu, and B.~Wong.
\newblock On the canonical line bundle and negative holomorphic sectional
  curvature.
\newblock {\em Math. Res. Lett.}, 17(6):1101--1110, 2010. MR2729634, Zbl 1233.14005.

\bibitem[Igu54]{Igusa}
J.~Igusa.
\newblock On the structure of a certain class of {K}aehler varieties.
\newblock {\em Amer. J. Math.}, 76:669--678, 1954. MR0063740, Zbl 0058.37901.

\bibitem[Kaw81]{kawamata_81}
Y.~Kawamata.
\newblock Characterization of abelian varieties.
\newblock {\em Compositio Math.}, 43(2):253--276, 1981. MR0622451, Zbl 0471.14022.

\bibitem[Kaw85a]{Kawamata_85}
Y.~Kawamata.
\newblock Minimal models and the {K}odaira dimension of algebraic fiber spaces.
\newblock {\em J. Reine Angew. Math.}, 363:1--46, 1985. MR0814013, Zbl 0589.14014.

\bibitem[Kaw85b]{Kawamata_85_pluri_sys}
Y.~Kawamata.
\newblock Pluricanonical systems on minimal algebraic varieties.
\newblock {\em Invent. Math.}, 79(3):567--588, 1985. MR0782236, Zbl 0593.14010.

\bibitem[Kle66]{Kleiman}
S.~Kleiman.
\newblock Toward a numerical theory of ampleness.
\newblock {\em Ann. of Math. (2)}, 84:293--344, 1966. MR0206009, Zbl 0146.17001.

\bibitem[KM98]{KM}
J.~Koll{\'a}r and S.~Mori.
\newblock {\em Birational geometry of algebraic varieties}, volume 134 of {\em
  Cambridge Tracts in Mathematics}.
\newblock Cambridge University Press, Cambridge, 1998.
\newblock With the collaboration of C. H. Clemens and A. Corti, Translated from
  the 1998 Japanese original. MR1658959, Zbl 0926.14003.

\bibitem[KN69]{kobayashi_nomizu_ii}
S.~Kobayashi and K.~Nomizu.
\newblock {\em Foundations of differential geometry. {V}ol. {II}}.
\newblock Interscience Tracts in Pure and Applied Mathematics, No. 15 Vol. II.
  Interscience Publishers John Wiley \& Sons, Inc., New York-London-Sydney,
  1969. MR0238225, Zbl 0175.48504.

\bibitem[KV80]{Kawamata_Viehweg_char_ab_var}
Y.~Kawamata and E.~Viehweg.
\newblock On a characterization of an abelian variety in the classification
  theory of algebraic varieties.
\newblock {\em Compositio Math.}, 41(3):355--359, 1980. MR0589087, Zbl 0417.14033.

\bibitem[Laz04]{Laz_I}
R.~Lazarsfeld.
\newblock {\em Positivity in algebraic geometry. {I}}, volume~48 of {\em
  Ergebnisse der Mathematik und ihrer Grenzgebiete. 3. Folge.}
\newblock Springer-Verlag, Berlin, 2004. MR2095471, Zbl 1093.14501.

\bibitem[Liu14]{liu}
G.~Liu.
\newblock Compact {K}\"ahler manifolds with nonpositive bisectional curvature.
\newblock {\em Geom. Funct. Anal.}, 24(5): 1591--1607, 2014. MR3261635.

\bibitem[Lu02]{Lu02}
S.~Lu.
\newblock A refined {K}odaira dimension and its canonical fibration.
\newblock {\em {\rm arXiv:math/ 0211029v3}}, 2002.

\bibitem[MM83]{MM}
S.~Mori and S.~Mukai.
\newblock The uniruledness of the moduli space of curves of genus {$11$}.
\newblock In {\em Algebraic geometry ({T}okyo/{K}yoto, 1982)}, volume 1016 of
  {\em Lecture Notes in Math.}, pages 334--353. Springer, Berlin, 1983. MR0726433, Zbl 0557.14015.

\bibitem[MP97]{Miyaoka_Peternell_book}
Y.~Miyaoka and Th. Peternell.
\newblock {\em Geometry of higher-dimensional algebraic varieties}, volume~26
  of {\em DMV Seminar}.
\newblock Birkh\"auser Verlag, Basel, 1997. MR1468476, Zbl 0865.14018.

\bibitem[Pet91]{peternell}
Th.~Peternell.
\newblock Calabi-{Y}au manifolds and a conjecture of {K}obayashi.
\newblock {\em Math. Z.}, 207(2):305--318, 1991. MR1109668, Zbl 0735.14028.

\bibitem[Rat06]{JR}
J.~Ratcliffe.
\newblock {\em Foundations of hyperbolic manifolds}, volume 149 of {\em
  Graduate Texts in Mathematics}.
\newblock Springer, New York, second edition, 2006. MR2249478, Zbl 1106.51009.

\bibitem[Roy80]{Royden_80}
H.~Royden.
\newblock The {A}hlfors-{S}chwarz lemma in several complex variables.
\newblock {\em Comment. Math. Helv.}, 55(4):547--558, 1980. MR0604712, Zbl 0484.53053.

\bibitem[Shi71]{Shiffman_ext}
B.~Shiffman.
\newblock Extension of holomorphic maps into hermitian manifolds.
\newblock {\em Math. Ann.}, 194:249--258, 1971. MR0291507, Zbl 0213.36001.

\bibitem[Siu11]{Siu_AC}
Y.-T. Siu.
\newblock Abundance conjecture.
\newblock In {\em Geometry and analysis. {N}o. 2}, volume~18 of {\em Adv. Lect.
  Math. (ALM)}, pages 271--317. Int. Press, Somerville, MA, 2011. MR2882447, Zbl 1263.14018.

\bibitem[Tsu00]{Tsuji}
H.~Tsuji.
\newblock Numerically trivial fibrations.
\newblock {\em {\rm arXiv:math.AG/0001023}}, 2000.

\bibitem[TY15]{tosatti_yang}
V.~Tosatti, X.~Yang.
\newblock An extension of a theorem of Wu-Yau.
\newblock {\em {\rm arXiv:1506.01145v2}}, 2015.

\bibitem[Uen75]{Ueno}
K.~Ueno.
\newblock {\em Classification theory of algebraic varieties and compact complex
  spaces}.
\newblock Lecture Notes in Mathematics, Vol. 439. Springer-Verlag, Berlin,
  1975.
\newblock Notes written in collaboration with P. Cherenack. 
MR0506253, Zbl 0299.14007.

\bibitem[Wu73]{Wu}
H.~Wu.
\newblock A remark on holomorphic sectional curvature.
\newblock {\em Indiana Univ. Math. J.}, 22:1103--1108, 1972/73. MR0315642, Zbl 0265.53055.

\bibitem[WWY12]{wong_wu_yau}
P.-M.~Wong, D.~Wu, and S.-T. Yau.
\newblock Picard number, holomorphic sectional curvature, and ampleness.
\newblock {\em Proc. Amer. Math. Soc.}, 140(2):621--626, 2012. MR2846331, Zbl 1235.32021.

\bibitem[WY15]{wu_yau}
D.~Wu, S.-T.~Yau.
\newblock Negative holomorphic curvature and positive canonical bundle.
\newblock {\em {\rm arXiv:1505.05802}}, 2015.

\bibitem[WZ02]{Wu_Zheng_JDG}
H.~Wu and F.~Zheng.
\newblock Compact {K}\"ahler manifolds with nonpositive bisectional curvature.
\newblock {\em J. Differential Geom.}, 61(2):263--287, 2002. MR1972147, Zbl 1071.53539.

\bibitem[Yau78]{Yau}
S.-T.~Yau.
\newblock On the {R}icci curvature of a compact {K}\"ahler manifold and the
  complex {M}onge-{A}mp\`ere equation. {I}.
\newblock {\em Comm. Pure Appl. Math.}, 31(3):339--411, 1978. MR0480350, Zbl 0369.53059.

\bibitem[Zhe93]{zheng_annals}
F.~Zheng.
\newblock Non-positively curved {K}\"ahler metrics on product manifolds.
\newblock {\em Ann. of Math. (2)}, 137(3):671--673, 1993. MR1217351, Zbl 0779.53045.

\bibitem[Zhe95]{Zheng_surfaces_london}
F.~Zheng.
\newblock On compact {K}\"ahler surfaces with non-positive bisectional
  curvature.
\newblock {\em J. London Math. Soc. (2)}, 51(1):201--208, 1995. MR1310732, Zbl 0815.53076.

\bibitem[Zhe02]{Zheng_helv}
F.~Zheng.
\newblock Kodaira dimensions and hyperbolicity of nonpositively curved compact
  {K}\"ahler manifolds.
\newblock {\em Comment. Math. Helv.}, 77(2):221--234, 2002. MR1915039, Zbl 1010.53054.

\end{thebibliography}
\end{document}